\documentclass[12pt,reqno]{amsart}

\usepackage{amssymb,array}
\usepackage{amsmath}
\usepackage{amsthm}
\usepackage{amsbsy,mathrsfs}
\usepackage{bm}
\usepackage[english]{babel}
\usepackage{graphicx}
\usepackage{color}
\usepackage{graphicx}
\usepackage{color}
\usepackage{hyperref}
\usepackage[toc,page]{appendix}
\usepackage{thmtools}
\usepackage{thm-restate}
\usepackage{textcase}

\usepackage{cleveref}

\theoremstyle{plain}
\newtheorem{theorem}{Theorem}[section]
\newtheorem*{theorem*}{Theorem}
\newtheorem{lemma}[theorem]{Lemma}

\newtheorem{corollary}[theorem]{Corollary}

\theoremstyle{definition}
\newtheorem{definition}[theorem]{Definition}

\newcommand{\kau}{\mathscr{K}}
\newcommand{\ask}{\mathscr{A}}
\newcommand{\al}{\mathfrak{a}}
\newcommand{\gl}{\mathfrak{g}}
\newcommand{\g}{\gamma}

\newcommand{\C}{\mathbb{C}}

\newcommand{\Z}{\mathbb{Z}}

\newcommand{\R}{\mathbb{R}}
\newcommand{\U}{\mathcal{U}}
\newcommand{\s}{\mathcal{S}}
\newcommand{\pa}{\mathscr{P}}
\newcommand{\e}{\epsilon}

\renewcommand{\H}{{\mathbb{H}}}

\newcommand{\la}{\langle}
\newcommand{\ra}{\rangle}
\newcommand{\x}{\widetilde{x}}
\newcommand{\y}{\widetilde{y}}

\newcommand{\tpi}{\widetilde{\pi}}
\newcommand{\M}{\mathcal{M}}
\newcommand{\T}{\mathcal{T}}

\newcommand\conju[1]{\widetilde{#1}}

\newcommand{\G}{\mathcal{G}}
\newcommand{\GW}{\mathcal{GW}}

\theoremstyle{plain} 
\newcommand{\thistheoremname}{}
\newtheorem*{genericthm*}{\thistheoremname}

\begin{document}
	\title[Center of Poisson and skein]{Center of Poisson and skein algebras associated to loops on surfaces
	}
	
	\author{ Arpan Kabiraj}
	\address{Department of Mathematics, Indian Institute of Technology Palakkad}
	\email{arpaninto@iitpkd.ac.in}
	
	\begin{abstract}
	
	  We discuss and develop a systematic method to compute the Poisson center (Casimir)  of various Poisson algebras associated to loops on orientable surfaces (possibly with boundary and punctures) introduced by Goldman and Wolpert in 80's while studying Thurston's earthquakes deformations.  Our computation extends a result of Etingof to all finite type hyperbolic surfaces.
	  We use these methods to compute the center of various skein algebras introduced by Turaev for the quantization of these Poisson algebras.  As another application of our results we compute the center of homotopy skein algebra introduced by Hoste and Przytycki. 
	\end{abstract}
	\maketitle
	
	\section{Introduction}
	
Let $\Sigma$ be an oriented surface, possibly with boundary and punctures. We assume that the Euler characteristic of $\Sigma$ is negative so that $\Sigma$ admits a hyperbolic metric. We call such a surface a \emph{hyperbolic surface}. We denote the fundamental group of $\Sigma$ by $\pi_1(\Sigma)$ and the set of all free homotopy classes of oriented (respectively unoriented) closed curves  in $\Sigma$ by $\pi$ (respectively $\tpi$).  Unless otherwise specified, we assume $K$ to be a commutative ring with identity containing the ring of integers $\Z$.  
For any set $S$, we denote by $KS$ the free $K$-module generated by $S$.

In \cite{goldman_invariant_1986}, Goldman discovered a Lie bracket on $K\pi$ while studying the symplectic structure on the moduli space of representations of $\pi_1(\Sigma)$.
 The module $K\pi$ with this Lie bracket is known as the \emph{Goldman Lie algebra} which we denote by $\G$. He also showed that $\G$ admits a natural Lie subalgebra on the submodule $K\tpi$.
 This Lie subalgebra is known as the \emph{Thurston-Wolpert-Goldman Lie algebra} which we denote by $\GW$.  See Section \ref{sec:prel} for definitions and details.

Goldman's work was inspired from the work of Wolpert (\cite{wolpert1983symplectic}, \cite{wolpert_fenchel-nielsen_1982}). In fact the Lie algebra $\GW$ was implicit in \cite{wolpert1983symplectic} (Wolpert called it \emph{twist lattice}).

The symmetric algebras $\s (\G )$ and $\s (\GW )$ of $\G$ and $\GW$ respectively, admit natural Poisson algebra structures. If we allow Poisson algebras to be non-commutative then the universal enveloping algebras $\U (\G )$ and $\U (\GW )$ also become Poisson algebras.  Following Turaev, we (informally) call these Poisson algebras, the \emph{Poisson algebras of loops} on $\Sigma$.

In \cite{turaev1991skein}, Turaev introduced various skein algebras associated to the set of isotopy classes of links in the  three manifold $\Sigma\times I$ for the quantization of the Poisson algebras of loops.
In \cite{hoste1990homotopy}, Hoste and Przytycki  independently gave another quantization of the Poisson algebras of loops in terms of \emph{homotopy skein algebras}.
For recent development in the Poisson structure of the moduli space of representations and various skein algebras associated to them, see \cite{turaev2019loops}, \cite{turaev2019topological}, \cite{turaev2020quasi}, \cite{charles2009multicurves}, \cite{marche2011kauffman}, \cite{MR3950651}, \cite{MR3709650}, \cite{MR2326938}, \cite{MR3885180}, \cite{marche2021valuations}.

The goal of this paper is twofold. Firstly to compute the \emph{Poisson center} of the Poisson algebras of loops mentioned above. Secondly to use the quantizations of Turaev \cite{turaev1991skein} and Hoste and Przytycki \cite{hoste1990homotopy} to compute the center of various skein algebras. 

Now we state the main results of the paper. The Poisson center $\mathcal{Z}(A)$, of a Poisson algebra $(A,\{,\})$ is the subalgebra consists of elements $y\in A$ such that $\{x,y\}=0$ for all $x\in A$. Elements of $\mathcal{Z}(A)$ are also known as \emph{Casimir elements}. The center $\mathcal{Z}(A)$ plays an important role in the study of representations of $A$. 

In an earlier paper \cite{chas2020lie} with Moira Chas, the author computed the Poisson center of the Poisson algebras $\U(\GW)$ and $\s(\GW)$. In the first main result of this paper we compute the Poisson center of the Poisson algebras $\U(\G)$ and $\s(\G)$ generalizing the result of Etingof \cite{etingof2006casimirs} to all finite type hyperbolic surfaces.

\begin{theorem*}
	For a hyperbolic surface $\Sigma$, the Poisson center of the Poisson algebras $\U(\G)$ and $\s(\G)$
	are generated by three types of homotopy classes of oriented curves: 1) constant curve, 2) curves homotopic to boundaries and 3) curves homotopic to punctures.   
\end{theorem*} 

The above theorem naturally extends to the one parameter family $\s_k(\G)$ of Poisson algebras associated to $\s(\G)$ introduced by Turaev in \cite[Section 2.2]{turaev1991skein} (see Section \ref{subsec:S_k}).

\vspace*{2mm}

The skein algebras associated to $\Sigma \times  I$ play a central role in the study of quantum invariants, topological quantum field theory and finite type invariants of links in $3$-manifolds. In a series of papers \cite{MR3480556},  \cite{MR3709650}, \cite{MR3950651}  Bonahon and Wong studied the finite dimensional representations of the Kauffman bracket skein algebra which gave a quantization of the $SL_2(\C)$ character variety. The central elements of  Kauffman bracket skein algebra played a key role in their study and provided ``classical shadow" invariant for the representations. Subsequently the center of the Kauffman bracket skein algebra was studied in \cite{MR3342686}, \cite{MR3885180}.   

 In \cite{turaev1991skein}, Turaev introduced quantization of Poisson algebras of loops in terms of Skein algebras $\ask(\Sigma)$, {\bf A}($\Sigma$), $\kau(\Sigma)$ and their quotients (see Section \ref{sec:orient} and Section \ref{sec:unorient} for details). For oriented curves, the skein relations of these skein algebras are similar to the Jones-Conway skein relations. For unoriented curves, the skein relations are modified version of Kauffman bracket skein relations. 
  In \cite{hoste1990homotopy},  Hoste and Przytycki gave another quantization of Poisson algebra of loops in terms of homotopy skein algebras $ \mathscr{H}\mathscr{S}(\Sigma\times I)$ and $\mathscr{KH}\mathscr{S}(\Sigma\times I)$ (see Section \ref{sec:homskein} for details). 
  Our aim is to compute the centers of these skein algebras.

In that direction, the main results proved in the second part of the paper can be summarized in the following theorem. 

\begin{theorem*}
	The centers of the following skein algebras (over appropriate rings)
	\\$\ask(\Sigma)/\hbar\ask(\Sigma)$,$\ask(\Sigma)/((x-1)\ask(\Sigma)+(h-1)\ask(\Sigma)+\hbar \ask(\Sigma))$,  {\bf A}($\Sigma$)/$\hbar$ {\bf A}($\Sigma$),\\ $\kau(\Sigma)/((x-1)\kau(\Sigma)+h_{-1}\kau(\Sigma) )$, $\kau(\Sigma)/(h_0\kau(\Sigma)+h_{-1}\kau(\Sigma) )$%
\\	 are generated by the empty link, the constant link and the  links whose projection to $\Sigma$ are isotopic to the boundary components or punctures of $\Sigma$.
\end{theorem*}

As an application of our computation of Poisson centers we also obtain the following result.  

\begin{theorem*}
	The center of the homotopy skein algebra of oriented (respectively unoriented) links is generated by the empty link, the constant link and the oriented (respectively unoriented) links which are link homotopic to the boundary components or punctures of $\Sigma$.
\end{theorem*}

We further exploit the relation between these Poisson algebras and homotopy skein algebras to obtain various properties of these skein algebras. 

\subsection{Organization of the paper} In Section \ref{sec:prel}, we recall the definitions and properties of the Poisson algebras of loops introduced in \cite{goldman_invariant_1986} and related Poisson algebras introduced in \cite{turaev1991skein} and \cite{hoste1990homotopy}. In Section \ref{sec:hyp}, we recall some basic facts about  hyperbolic surfaces. In Section \ref{sec:lift} we describe the geometry of the lifts of the terms of various Lie brackets of these Poisson algebras. In Section \ref{sec:tech}  we prove the main technical lemmas required for our proof. In Section \ref{sec:univ} we compute the center of the Poisson algebras of loops. In Section \ref{sec:orient},\ref{sec:unorient} and \ref{sec:homskein} we define the various skein algebras and compute their centers. In Section \ref{sec:spec} we explain a possible generalization of our theorems.  
\subsection*{Acknowledgment}  The author would like to thank Prof. Moira Chas for her encouragement and enlightening conversations. 
   
\section{Poisson algebras of loops}\label{sec:prel}
Let $\Sigma$ be a hyperbolic surface. 
We consider both oriented and unoriented closed curves on $\Sigma$. Given an oriented curve $x$ we denote the corresponding unoriented curve by $\x$. 

 We fix an orientation on $\Sigma$ once and for all. 
 All the pictures are drawn assuming the anti-clockwise orientation of the surface.

There is a one-to-one correspondence between the set of all free homotopy classes of oriented closed curves in $\Sigma$ and the set of all conjugacy classes in $\pi_1(\Sigma)$, both of which will be denoted by $\pi$. The set of all free homotopy classes of unoriented closed curves will be denoted by ${\tpi}$.
 
 We denote the free homotopy class of an oriented (respectively unoriented) closed curve $x$ (respectively $\x$) by $\la x\ra$ (respectively $\la \x\ra$).

We start with the definition of Goldman Lie algebra and Thurston-Wolpert-Goldman Lie algebra.

\begin{definition}
	Consider two oriented closed curves $x$ and $y$ in $\Sigma$ which intersect each other transversally in double points. Define the \emph{Goldman bracket} between $\la x\ra$ and $\la y\ra$ by 
	$$[\la x\ra,\la y\ra]=\sum_{p\in x\cap y}\epsilon_p\la x*_py\ra,$$
	where $x\cap y$ denotes the set of all intersection points between $x$ and $y$, $\epsilon_p$ denotes the sign of the intersection between $x$ and $y$ at $p$ and  $x*_py$ is the loop product of $x$ and $y$ at $p$. Extend the Goldman bracket to $K\pi$ linearly.

\end{definition}

Goldman \cite{goldman_invariant_1986} proved that the bracket is well defined on $K\pi$, is antisymmetric and satisfy Jacobi identity. Hence $K\pi$ with the Goldman bracket forms a Lie algebra called the \emph{Goldman Lie algebra} and it is denoted by $\G$. For convenience of notation we denote $[\la x \ra,\la y\ra]$ simply by $[x,y]$. 

Consider the involution $i:\pi\rightarrow \pi$ defined by $i(x)=x^{-1}$ where $x^{-1}$ denotes the same curve $x$ but with opposite orientation. The map $i$ extends linearly to a Lie algebra automorphism of $\G$. Therefore the stationary set $\{x\in K\pi:i(x)=x\}$ is a Lie subalgebra of $\G$. The stationary set is canonically isomorphic to $K\tpi$. We call this Lie subalgebra the \emph{Thurston-Wolpert-Goldman Lie algebra} and denote it by $\GW.$

 For an intrinsic definition of Thurston-Wolpert-Goldman Lie algebra see \cite{chas2020lie}. For convenience of notation we denote $[\la \x \ra,\la \y\ra]$ simply by $[\x,\y]$.

Let $\U(\G)$ be the universal enveloping algebra and $\s(\G)$ be the symmetric algebra of $\G$. Similarly denote the universal enveloping algebra and symmetric algebra of $\GW$ by  $\U(\GW)$ and $\s(\GW)$ respectively.
For definition and basic properties of universal enveloping algebra and symmetric algebra  see \cite{abe2004hopf}, \cite{hoste1990homotopy}.

A \emph{Poisson algebra} over $K$ is an associative $K$-algebra (possibly non-commutative) together with a Lie bracket that also satisfies Leibniz's rule. The algebras $\U(\G)$ and $\U(\GW)$ have natural Poisson algebra structures with the commutator being the Lie bracket. We extend the Lie bracket of $\G$ and $\GW$ - using Leibniz rule - to $\s(\G)$ and $\s(\GW)$ respectively. This make $\s(\G)$ and $\s(\GW)$ Poisson algebras. 

There are canonical maps from $\G$ and $\GW$ to their corresponding universal enveloping algebra and symmetric algebra. To simplify notation, we denote an element and its image under these maps by the same notation. We also denote the product of two elements  in  these Poisson algebras simply by juxtaposing the elements.

Let us recall the Poincare-Birkhoff-Witt theorem for $\G$ and $\GW$ (see \cite{abe2004hopf}, \cite[Theorem 3.2]{hoste1990homotopy}). We state it for $\G$. The statement for $\GW$ is exactly the same after you replace the oriented curves with their corresponding unoriented representatives.
\begin{theorem}\label{thm:pbw}
	Let $\leq $ be a fixed total order on ${\pi}$. Consider the set
	$$S=\{{x}_1 {x}_2\cdots  {x}_n:{x}_i\in{\pi}, i\in\{1,2,\ldots ,n\},  {x}_1\leq {x}_2\leq \cdots  \leq {x}_n\}.$$ 
	Both  $\U(\G)$  and $\s(\G)$ are freely generated by $S$ as $K$ modules. Moreover the natural maps from ${\G}$ into  $\U(\G)$  and $\s(\G)$ are injective Lie algebra homomorphisms. 
\end{theorem}

\subsection{The algebras $V_h(\G)$ and $V_h(\GW)$}\label{subsec:V} Let $\gl$ be any Lie algebra over $K$. We recall the description of the algebra $V_h(\gl)$ and its one parameter variants introduced by Turaev in \cite[Section 1.3]{turaev1991skein}. 

 Let $\al$ be the $K[h]$-module $K[h]\otimes \gl$ and $T(\al)$ be the tensor algebra of $\al$. The algebra  $V_h(\gl)$ is defined to be the quotient of $T(\al)$  by the two sided ideal generated by the relations $\{xy-yx-h[x,y]:x,y\in  \gl \subset \al\}$. Therefore $V_h(\gl)$ is an an associative $K[h]$ algebra generated by $\gl$.

 Consider the Lie algebras $K[h]\otimes \G$ and $K[h]\otimes \GW$ over $K[h]$, where we replace the Lie brackets by $h$ times the Goldman Lie bracket and Thurston-Wolpert-Goldman Lie bracket respectively. By {\cite[Remarks 2, Section 4.6]{turaev1991skein}}, $V_h(\G)$ and $V_h(\GW)$ are isomorphic to $\U(K[h]\otimes \G)$ and $\U(K[h]\otimes \GW)$ respectively.  Therefore $V_h(\G)$ (respectively $V_h(\GW)$) is freely generated by the set $S$ (respectively the same set $S$ with oriented curves replaced by unoriented curves) in Theorem \ref{thm:pbw} as a $K[h]$ module.
 
 Turaev also considered the one parameter family of $K$-algebras $V_h^k(\gl)=\{V_h(\gl)/(h-k)V\}$ for any $k\in K$. When $K$ is a field and $k\neq 0,$ each $V_h^k(\gl)$ is isomorphic to the algebra $\U(\gl).$ 
   
\subsection{A family $\s_k(\G)$ of deformations of $\s(\G)$}\label{subsec:S_k}
 Given any $k\in K$, we define $[x,y]_k$ by $$[x,y]_k=[x,y]-k(x \cdot y)xy$$ for all $x,y \in \pi$, where $(x\cdot y)$ is the algebraic intersection number between $x$ and $y$. Extend $[x,y]_k$ to $\G$ using bi-linearity and  to $\s(\G)$ using Leibniz rule.  Then $[\,\,,]\,_k$ satisfy the Jacobi identity (see \cite[Section 2.2]{turaev1991skein}).
We define the Poisson algebra $\s_k(\G)$ over $K$ to be the $K$-algebra $\s(\G)$ with the Lie  bracket $[\,\,,]\,_k$. Observe that $\s_0(\G)=\s(\G)$.
\section{Hyperbolic geometry on surface.} \label{sec:hyp}
 
In this section, we recall some basic results about hyperbolic geometry on surfaces. For the proofs of the results see \cite{beardon2012geometry}, \cite{MR2850125} and the references therein. 

Let $\Sigma$ be an oriented surface of genus $g$ with $b$ boundary components and $n$ punctures. We assume the Euler characteristic of $\Sigma$ to be negative, i.e. $2-2g-b-n<0$.  Every such surface admits a hyperbolic metric. We call an oriented surface with a given hyperbolic metric a \emph{hyperbolic surface}.

Given any hyperbolic surface, we identify its universal cover with the hyperbolic plane $\H$ and its fundamental group $\pi_1(\Sigma)$ with a discrete subgroup of $PSL_2(\R)$,  the group of orientation preserving isometries of $\H$.

A closed curve is called \emph{essential} if it is not homotopic to a point or to a puncture. It is called \emph{peripheral} if it is homotopic to a puncture. 


Given any essential oriented (respectively unoriented) closed curve $x$ in a hyperbolic surface, there is a unique oriented (respectively unoriented) geodesic in $\la x\ra$. Unless otherwise specified, we use the geodesic representatives from the corresponding free homotopy classes.

Recall that the Teichm\"{u}ller space $\T$ of $\Sigma$ is the set of all  isotopy classes of marked hyperbolic surfaces $(F,\phi)$ where $F$ is a hyperbolic surface and $\phi:\Sigma\rightarrow F$ is an orientation preserving diffeomorphism. Given any point $(F,\phi)\in \T$, we associate a natural hyperbolic metric on $\Sigma$, the pull back of the metric on $F$ by $\phi$. We use this association to identify a point $(F,\phi)\in\T$ as a hyperbolic metric on $\Sigma$. 
For ease of notation we sometimes denote $(F,\phi)$ simply by $F$. 
Given any free homotopy class $\la x\ra$  and any point $F\in \T$, let $l_x(F)$ be the \emph{length function}, i.e., the length of the unique geodesic in $\la x\ra$ in $\Sigma$ corresponding to the hyperbolic metric $F$. We call a geodesic in the metric $F$ an \emph{$F$-geodesic}.  

Let $F\in\T$ be a hyperbolic metric on $\Sigma$ and $x,y$ be two oriented  $F$-geodesics intersecting transversally at an intersection point $p$. 
We define $\theta_p(F)$  to be the angle at $p$ between the positive direction of $x$ and $y$.
To keep notations simple, we sometimes omit $F$ from the notations of angle and length. The meaning should be clear from the context.

Given two closed curves $x$ and $y$, the \emph{geometric intersection number} $i(x,y)$ between $x$ and $y$ is defined to be $i(x,y)=min\{a\cap b:a\in \la x \ra,b\in \la y\ra\}. $  
The following lemma is a  well known result.  
\begin{lemma}\label{lem:classical}
	Let $\Sigma$ be an orientable surface 
	and  $y$ be a closed curve on $\Sigma$ such that $i(x,y)=0$ for every simple closed curve $x$. Then $y$ is either homotopically trivial or homotopic to a boundary curve or homotopic to a puncture. 	
\end{lemma}

\section{Lift of the loop products to $\H$}\label{sec:lift}
In this section we discuss the geometry of loop product described in the definition of Goldman Lie bracket. For details see \cite{chas2016extended}, \cite{kabiraj2016center}.

We fix a metric $F\in\T$. We discuss all geometric objects, lengths and angles with respect to $F$ without mentioning it explicitly. We identify the universal cover of $\Sigma$ (with metric $F$) to $\H$. Unless otherwise mentioned, in this section, we assume all geodesics to be oriented.  

Let $x$ and $y$ be two oriented geodesics in $\Sigma$ and $p$ be a transverse intersection double point between them. We call any oriented lift of a geodesic $x$ to $\H$ an \emph{axis} of $x$. We use the notation $A_x$ for a generic axis of $x$ and $A_x(P)$ for an axis of $x$ passing through a point $P\in \H$.
Our starting point is the following result to find an axis and the length of the geodesic in $\la x*_py\ra $. 

Let $P$ be any lift of $p$ to $\H$. There exist two lifts $A_x(P)$ and $A_y(P)$ of $x$ and $y$ respectively, intersecting at $P$ (see Figure \ref{axis}). Let $R$ be the point on $A_x(P)$ at a distance $l_x/2$ from $P$ in the forward direction of $A_x(P)$ and $S$ be the point on $A_y(P)$ at a distance $l_y/2$ from $P$ in the backward direction of $A_y(P)$.

\begin{figure}[h]
	\centering
	\includegraphics[trim = 30mm 25mm 50mm 10mm, clip, width=10cm]{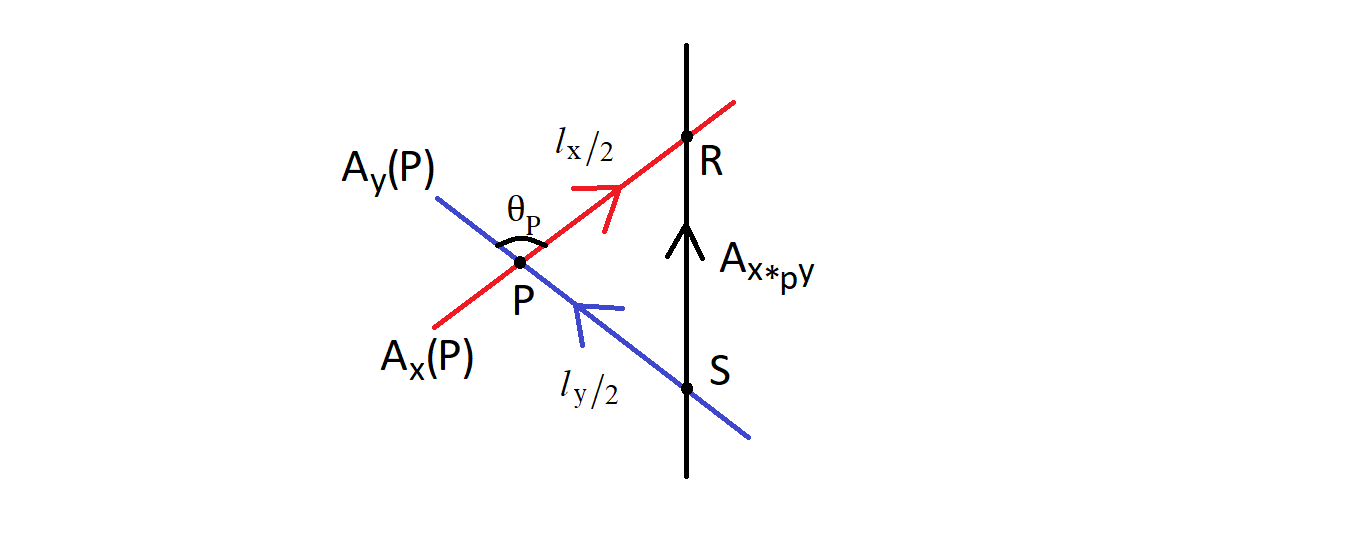}
	\caption{~}\label{axis}
\end{figure}

\begin{theorem}{\cite[Theorem 7.38.6]{beardon2012geometry}} \label{thm:beardon}
With the above notation, the geodesic containing the geodesic segment from $S$ to $R$ (with orientation from $S$ to $R$) is an axis of the geodesic in  $\la x*_py\ra$. We also have 
$$\cosh\left(\frac{l_{x*_py}}{2}\right)=
\cosh\left(\frac{l_x}{2}\right)\cosh\left(\frac{l_y}{2}\right)+
\sinh\left(\frac{l_x}{2}\right)\sinh\left(\frac{l_y}{2}\right)\cos\theta_p.$$
\end{theorem}

\subsection*{Construction of a lift}
We recall the construction of a lift of  $(x*_py)$ from  \cite[Section 7]{chas2016extended}. Let $P_0$ be a lift of $p$. To construct a lift of $(x*_py)$ passing through $P_0$, we do the following (see Figure \ref{zigzag}). First consider the geodesic axis $A_x(P_0)$ of $x$. Travel the length $l_x$, from $P_0$, in the forward direction along  $A_x(P_0)$ to reach the point $P_1$. Now from $P_1$ travel the distance $l_y$ along $A_y(P_1)$ to reach $P_2$. We continue the same process from $P_2$ and repeat it indefinitely. Similarly we do the same construction from $P_0$ in backward direction. The construction yields a lift of $(x*_py)$ which is a  bi-infinite piecewise geodesic whose geodesic pieces are consecutive geodesic arcs of axes of $x$ and $y$ of length $l_x$ and $l_y$ respectively.

\begin{figure}[h]
	\centering
	\includegraphics[trim = 300mm 75mm 350mm 10mm, clip, width=6cm]{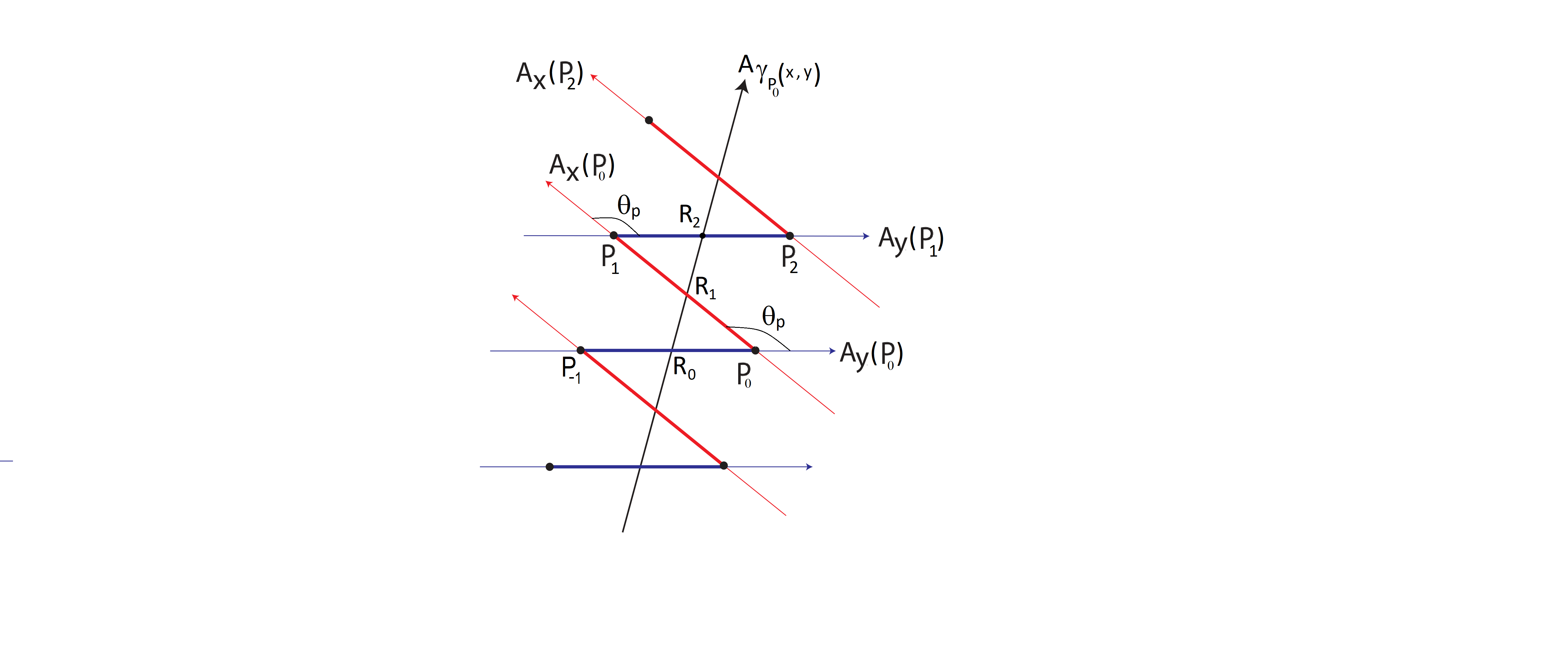}
	\caption{~}\label{zigzag}
\end{figure}

Let $R_{i}$ be the midpoint of $P_{i-1}$ and $P_{i}$ for all $i\in\Z$. Then by Theorem \ref{thm:beardon}, there exists an axis of $(x*_py)$  which passes through $R_i$'s.
We use the notation $\g_p(x,y)$ and $A_{\g_p(x,y)}$ for a generic lift of $x*_py$ and its axis obtained from the above construction. We call the geodesic arcs of $\g_p(x,y)$ 
corresponding to axes of $x$ and $y$ as \emph{$x$-pieces} and \emph{$y$-pieces} respectively. Therefore any lift $\g_p(x,y)$ obtained from the above construction is an oriented piecewise geodesic where the geodesic pieces are $x$-pieces and $y$-pieces appearing alternatively. Unless otherwise mentioned, by a lift of $(x*_py)$ we mean a lift obtained by the above construction. 

\begin{lemma}\label{lem:zigzag}
 Let $\g_p(x,y)$ be a lift of $x*_py$ and $A_{\g_p(x,y)}$ be its axis. Consider any two consecutive $x$-pieces (respectively $y$-pieces) of $\g_p(x,y)$ and let $R_t$ and $R_{t+2}$ be the intersection points between these pieces and $A_{\g_p(x,y)}$. Then the distance between $R_t$ and $R_{t+2}$ is $l_{x*_py}.$ Moreover there exist a $y$-piece (respectively $x$-piece) which intersects $A_{\g_p(x,y)}$ at $R_{t+1}$ which is the midpoint of $R_t$ and $R_{t+2}.$  
  
\end{lemma}
\begin{proof}
	The proof follows directly from the construction of $\g_p(x,y)$ (Figure \ref{zigzag}) and Theorem \ref{thm:beardon}. For details see \cite[Section 7]{chas2016extended}.  
\end{proof}

\section{Technical lemmas}\label{sec:tech}
In this section we prove the key lemmas that we use for proving our theorems. We prove them separately because they might be of independent interest. 

\begin{lemma}\label{lem:disjoint}
	Let $F\in \T$ and $x,y,z$ be three pairwise distinct oriented $F$-geodesics such that $x$ is simple. Let $p\in x\cap y$ and $q\in x\cap z$ such that $\theta_p=\theta_q$ and $n$ be any positive integer. Consider any two lifts  $\g_p(x^n,y)$ and $\g_q(x^n,z)$ of $(x^n*_py)$ and $(x^n*_qz)$ respectively. Suppose $\g_p(x^n,y)$ and $\g_q(x^n,z)$ have the same axis $A$.  Then for any $x$-piece $P_tP_{t+1}$ of $\g_p(x^n,y)$ and any $x$-piece $Q_sQ_{s+1}$ of $\g_q(x^n,z)$, we have $P_tP_{t+1}\cap Q_sQ_{s+1}=\emptyset$. 
\end{lemma}
\begin{proof}
	As $x$ is simple any two distinct geodesic lifts of $x$ are disjoint. Therefore any two $x$-pieces are either disjoint or intersect in a geodesic segment. 
	
	If possible suppose $P_tP_{t+1}$ and $Q_sQ_{s+1}$  intersect in a geodesic segment. By the construction $A$ intersects both $P_tP_{t+1}$ and $Q_sQ_{s+1}$ at their midpoints. Hence $P_tP_{t+1}=Q_sQ_{s+1}$. By the assumption $\theta_p=\theta_q$. Therefore - by construction - the $y$-piece of $\g_p(x^n,y)$ occurring immediately after $P_tP_{t+1}$ coincides with the $z$-piece of $\g_q(x^n,y)$ occurring immediately  after $Q_sQ_{s+1}$. This implies tha an axis of $y$ coincides with an axis of $z$ and $l_y=l_z$, which contradicts the fact that the geodesics $y$ and $z$ are distinct. 
\end{proof}

\begin{lemma}\label{lem:orienteql}
	Let $F\in \T$ and $x,y$ and $z$ be three oriented $F$-geodesics. Let $p$ be an intersection point between $x$ and $y$ and $q$ be an intersection point between $x$ and $z$. Suppose $\la x^m*_py\ra=\la x^m*_qz \ra$ for two distinct positive integral values of $m$. Then for any $F'\in \T$, $l_y(F')=l_z(F')$ and $\theta_p(F')=\theta_q(F')$.  
\end{lemma} 
\begin{proof}
	Fix any arbitrary metric $F'\in\T$. To simplify notation throughout the proof we measure lengths and angles with respect to $F'$ without mentioning it explicitly. Observe that the equality $\la x^m*_py\ra=\la x^m*_qz \ra$ is topological.
	
	 Now $l_{x^n}=nl_x$ for all positive integer $n$. As $\la x^m*_py\ra=\la x^m*_qy \ra$, the geodesics corresponding to the two free homotopy classes are the same. Hence $l_{ x^m*_py}=l_{x^m*_qz}$ for two distinct values of $m$. Therefore by Theorem \ref{thm:beardon}, 
	$$\cosh\left(\frac{ml_x}{2}\right)\cosh\left(\frac{l_y}{2}\right)+
	\sinh\left(\frac{ml_x}{2}\right)\sinh\left(\frac{l_y}{2}\right)\cos\theta_p$$$$ 
	 =\cosh\left(\frac{ml_x}{2}\right)\cosh\left(\frac{l_z}{2}\right)+
	\sinh\left(\frac{ml_x}{2}\right)\sinh\left(\frac{l_z}{2}\right)\cos\theta_q.$$
This implies 
$$ \coth(\frac{ml_{{x}}}{2})\{\cosh(\frac{ l_{{y}}}{2})-\cosh(\frac{ l_{{z}}}{2})\}=-\sinh(\frac{l_{{y}}}{2})\cos(\theta_p)+ \sinh(\frac{ l_{{z}}}{2} )\cos(\theta_q).	
$$	The left hand side of the equation depends on $m$ but the right hand side is independent of $m$. Hence as the equality holds for two distinct values of $m$, $\{\cosh(\frac{ l_{{y}}}{2})-\cosh(\frac{ l_{{z}}}{2})\}=0.$ Therefore $l_{y}=l_{z}$ and $\theta_p=\theta_q$. 
	
\end{proof}

\begin{lemma}\label{lem:signeql}
	Let $F\in \T$ and $x,y$ and $z$ be three oriented $F$-geodesics such that $x$ is simple. Let $p$ be an intersection point between $x$ and $y$ and $q$ be an intersection point between $x$ and $z$. Suppose $\la x^m*_py\ra=\la x^m*_qz \ra$ for two distinct positive integral values of $m$. Then $\e_p=\e_q$. 
\end{lemma} 

\begin{proof}
	If possible suppose $\e_p=-\e_q$. Without loss of generality assume $\e_p=1$. By Lemma \ref{lem:orienteql}, $\theta_p(F')=\theta_q(F')$ for all $F'\in\T$. By \cite[Lemma 2.1 and Lemma 6.1]{kabiraj2016center}, if we change the metric performing a left twist deformation along $x$, $\theta_p$ strictly increases but $\theta_q$ strictly decreases, which contradicts the equality of the angles for all metric in $\T$.   
\end{proof}

\begin{lemma}\label{lem:orientextang}
	Let $F\in \T$ and $x,y$ and $z$ be three pairwise distinct oriented $F$-geodesics. Let $p$ be an intersection point between $x$ and $y$ and $q$ be an intersection point between $x$ and $z$ such that $\e_p=\e_q$. Suppose $l_y=l_z$ and $x$ is simple. If for any positive integer $n$, $\la x^n*_py\ra=\la x^n*_qz \ra$  then one of the following is true:\\	
	1) either there exist $t_1\in x\cap y$ such that $\theta_{t_1}<\theta_p=\theta_q$,\\
	2) or there  exist $t_2\in x\cap z$ such that $\theta_{t_2}<\theta_p=\theta_q$.
\end{lemma}

\begin{proof}
	
	\begin{figure}[h]
		\centering
		\includegraphics[trim = 5mm 0mm 0mm 0mm, clip, width=13cm]{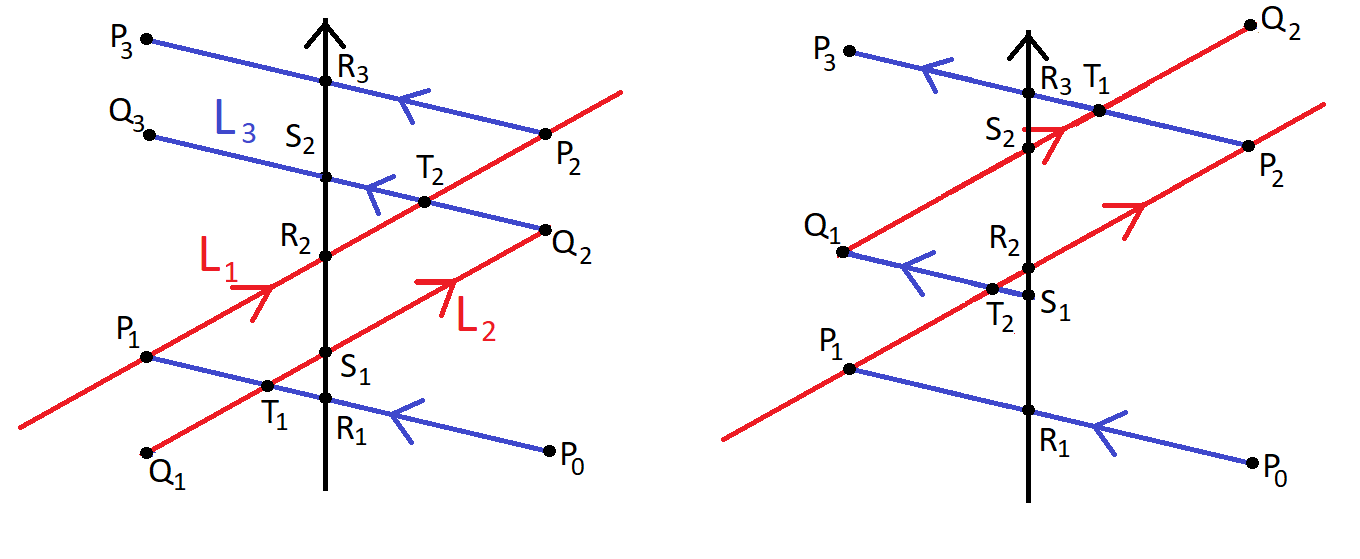}
		\caption{~}\label{angleproof}
	\end{figure}
	
		We prove the result for $n=1$. The angle at any intersection point $p$ between $x$ and $y$ is same as the angle at $p$ between $x^n$ and $y$ as they are physically the same point. Combining the last statement with the equality $l_{x^n}=nl_{x}$, the proof for $n>1$ follows by a similar argument.

	Without loss of generality assume $\e_p=1$. Fix a lift $\g_p(x,y)$ of $x*_py$ (see Figure \ref{angleproof}). As $\la x*_py\ra=\la x*_qz \ra$, both free homotopy classes have the same geodesic representative. Therefore $l_{x*_py}=l_{x*_qz}$. As $l_y=l_z$, by Theorem \ref{thm:beardon}, $\theta_p=\theta_q.$ 
	Also there exists a lift $\g_q(x,z)$ of $x*_qz$ such that $A_{\g_p(x,y)}=A_{\g_q(x,z)}$. We denote $A_{\g_p(x,y)}=A_{\g_q(x,z)}$ simply by $A$ and the length of an arc $RS$ by $l(RS)$.

	Choose a segment of $\g_p(x,y)$ consisting of two consecutive $y$-pieces $P_0P_1$ and $P_2P_3$ and an $x$-piece $P_1P_2$ between them. Let $R_i=P_{i-1}P_i\cap A$ for $i\in \{1,2,3\}$. By Lemma \ref{lem:zigzag}, $l(R_1R_3)=l_{x*_py}$.

	By the construction of  $\g_q(x,z)$, there exist an $x$-piece $Q_1Q_2$ of $\g_q(x,z)$ which intersects $A$ in the arc $R_1R_3$. There are two possibilities: Case (a) $Q_1Q_2$ intersects $A$  in $R_1R_2$ or Case (b) $Q_1Q_2$ intersects $A$ in  $R_2R_3$
	
{\bf\underline{Case (a) $Q_1Q_2$ intersects $A$  in $R_1R_2$:}}
 Consider left hand side picture of  Figure \ref{angleproof}. Let $Q_2Q_3$ be the $z$-piece in  $\g_q(x,z)$ occurring immediately after $Q_1Q_2$.    Let $L_1, L_2$ and $L_3$ be  the geodesics in $\H$ containing $P_1P_2,Q_1Q_2$ and $Q_2Q_3$ respectively. As $x$ is simple, either $L_1=L_2$ or $L_1\cap L_2=\emptyset$. If $L_1=L_2$ then $A\cap L_1=R_2=S_1=A\cap L_2$. Which implies $P_1P_2$ coincides with $Q_1Q_2$, contradicting Lemma \ref{lem:disjoint}. Therefore $L_1\cap L_2=\emptyset$, in particular $L_2\cap P_1P_2=\emptyset$. Consider the geodesic triangle $\triangle P_1R_1R_2$. As $L_2$ intersects $R_1R_2$ but does not intersect $P_1R_2$, $L_2$ must intersect $R_1P_1$. Denote the intersection point of $L_2$ and $R_1P_1$ by $T_1$. By construction, $Q_2Q_3$ intersects $A$ in $R_2R_3$. Therefore $Q_2Q_3$ must intersect $L_1$. Denote the intersection point between $Q_2Q_3$ and $L_1$ by $T_2$. Let $t_i$ be the projection of $T_i$ on $\Sigma$ for $i\in \{1,2\}$. Then $t_1\in x\cap y$ and $t_2\in x\cap z$.  Consider the geodesic quadrilateral  $\square T_1P_1T_2Q_2$. We have $\angle T_1+ \angle P_1+\angle T_2+\angle Q_2<2\pi$. Now $\angle P_1=\angle Q_2=\pi-\theta_p=\pi-\theta_q$, $\angle T_1=\theta_{t_1}$ and $\angle T_2=\theta_{t_2}$. Therefore we have $\theta_{t_1}+\theta_{t_2}<\theta_p+\theta_q$ which proves the claim.  

{\bf\underline{Case (b) $Q_1Q_2$ intersects $A$  in $R_2R_3$:}} 
 Consider right hand side picture of Figure \ref{angleproof} (right)). In this case we have to consider the $z$-piece $Q_0Q_1$ in  $\g_q(x,z)$ occurring before $Q_1Q_2$. By the same argument as above we get two points $T_1, T_2$ and their projections $t_1\in x\cap y$ and $t_2\in x\cap z$ with the desired property.    
\end{proof}

  \section{Universal enveloping algebra and symmetric algebra}\label{sec:univ}
\begin{theorem}\label{thm:univorient}
	The Poisson center of the Poisson algebras $\s(\G)$ and $\U(\G)$ are generated by scalars $K$, the free homotopy class of constant curve and the curves homotopic to boundaries and punctures.
\end{theorem}

\begin{proof}
	Fix a metric $F\in\T$ and choose ${x}$ to be any simple $F$-geodesic. Throughout the proof  we consider $F$-geodesic representatives of each curve.  First we prove the result for $\s(\G)$.
	 Let $Z$ be an element of the center of $\s(\G)$. Then by Theorem \ref{thm:pbw}, $$Z=\sum_{i=1}^m C_i\; {x}_{i_1}{x}_{i_2}\cdots {x}_{i_k}$$ where ${x}_{i_1}\leq {x}_{i_2}\leq \cdots  \leq {x}_{i_k}$ for all $i\in\{1,2,\ldots ,m\}$. We have 
	\begin{align*}
		0 &=[{x}^n,Z]\\
		 &=\sum_{i=1}^mC_i\;[x^n,{x}_{i_1}{x}_{i_2}\cdots {x}_{i_k}]\\
		&=\sum_{i=1}^{m}C_i\;
		\sum_{j=1}^{k}n
		\Bigg\{
		\sum_{p_{i_j}\in {x}\cap {x}_{i_j}}\e_{p_{i_j}}
		\Big(
		{x}_{i_1}\cdots {x}_{i_j-1}\la {x}^n*_{p_{i_j}}{x}_{i_j}\ra {x}_{i_j+1}\cdots {x}_{i_k}
		\Big)
		\Bigg\}.	
	\end{align*}
	
	Denote by $\bf{P}$, the set of all intersection points $p_{r_s}$ between $x$ and $x_{r_s}$ for all $r_s$.  
	Consider $p\in \bf{P}$ such that $\theta_p\leq \theta_q$ for all $q\in \bf{P}$. Let $p\in x\cap x_{i_j}$. 
	
	As $\s(\G)$ is commutative, we can rearrange the terms $\{{x}_{i_1}, \ldots  ,{x}_{i_j-1},\la {x}^n*_{p_{i_j}}{x}_{i_j}\ra, {x}_{i_j+1}, \ldots ,{x}_{i_k}\}$ in ascending order with respect to $\leq $ and assume they are elements of the set $S$ of Theorem \ref{thm:pbw}.    
	
	As $[{x}^n,Z]=0$ for all positive integer $n$, from the above expression and Theorem \ref{thm:pbw}, there exists ${x}_{k_l}\neq {x}_{i_j}$ such that  one of the following is true for all but  finitely many positive integers $n$.
	\begin{itemize}
		\item $\la {x}^n*_{p} {x}_{i_j}\ra =\la {x}^n*_{q}{x}_{k_l}\ra$ for some $q\in{x}\cap {x}_{k_l}$. In this case there are two possibilities:\\
		(1) $\e_p=\e_q$ which by Lemma \ref{lem:orientextang} contradicts the minimality assumption of angle $\theta_p.$ \\
		(2) $\e_p=-\e_q$ which is impossible by Lemma \ref{lem:signeql}.		
		\item $\la {x}^n*_{p}{x}_{i_j}\ra=\la {x}_{k_l}\ra$. This is impossible by Theorem \ref{thm:beardon} because the left hand side depends on $n$ but the right hand side is independent of $n$.	
	\end{itemize}  
	Therefore each ${x}_{i_j}$ is disjoint from ${x}$.  As ${x}$ is an arbitrary simple closed geodesic, each ${x}_{i_j}$ is disjoint from every simple closed geodesic on the surface. Hence by Lemma \ref{lem:classical} each ${x}_{i_j}$ is either a constant loop or a loop homotopic to a  puncture or a loop homotopic to a boundary component.



	Now we prove the result for $\U(\G)$. Let $T(\G)$ be the tensor algebra of $\G$. Let $\mathfrak{S}$ be the ideal generated by the elements of the form $x_1\otimes x_2\otimes \cdots \otimes x_n-x_{\sigma(1)}\otimes x_{\sigma(2)}\otimes \cdots \otimes x_{\sigma(n)}$, where $n$ is any positive integer and  $\sigma\in S_n$. Then $\s(\G)=T(\G)/\mathfrak{S}$.
	
	Let $\mathfrak{U}$ be the ideal generated by the elements of the form $x\otimes y-y\otimes x-[x,y]$. Then $\U(\G)=T(\G)/\mathfrak{U}$.
	
	Let $\Phi:\s(\U)\rightarrow \U(\G)$ be the canonical map. By Theorem \ref{thm:pbw}, $\Phi$ is a module isomorphism. 
	
	We have a $\G$ action on $\s(\G)$ obtained by extending the adjoint action of $\G$ on itself by derivations. For $x\in \G$ we denote the action on $\s(\G)$ by $S_x$.
	
	 On the other hand $\G$ also acts on $\U(\G)$ by the following: $\mathrm{for } \,\,  x\in \G \,\, \mathrm{ and } \,\, u\in \U(\G), U_x(u)=xu-ux .$
	 
	 A straightforward computation shows, for all $x\in \G, \Phi\circ S_x=U_{x}\circ \Phi.$ As the center of $\s(\G)$ is exactly $\s(G)^{\G}$ and the center of $\U(\G)$ is exactly $\U(\G)^{\G}$, the center of $\s(\G)$ and $\U(\G)$ are the same. Therefore the result follows from the first part.
	 
%

\end{proof}

The corresponding theorem for $\GW$ was proved in \cite{chas2020lie}. 

\begin{theorem}{\cite[Poisson Center
		Theorem]{chas2020lie}}\label{thm:univunorient}
	The Poisson center of the Poisson algebras $\U(\GW)$  and $\s(\GW)$ are generated by scalars $K$, the free homotopy class of constant curve and the curves homotopic to boundaries and punctures.
\end{theorem}

\begin{corollary}\label{cor:V}
	The center of $V_h(\G)$  and $V_h(\GW)$ are generated by scalars $K$, the free homotopy class of constant curve and the curves homotopic to boundaries and punctures. If $K$ is a field and $k\neq 0$ then the same result holds true for   $V_h^k(\G)$  and $V_h^k(\GW)$.
\end{corollary} 
\begin{proof}
	The proof follows from Section \ref{subsec:V} and from Theorem \ref{thm:univorient} and Theorem \ref{thm:univunorient}  by replacing $K$ with $K[h]$ and the brackets by $h$ times the brackets.
\end{proof}

\section{Skein algebras of oriented curves} \label{sec:orient}
Consider the  three manifold $\Sigma\times I$. A \emph{knot} in a three manifold is a smooth embedding of $S^1$ in the interior of the three manifold. A \emph{link} in a three manifold is a disjoint finite collection of knots. We also include the empty set as a unique link up to isotopy. Given any link $L$ we denote the number of components of $L$ by $|L|$. 

A triple of links $L_+,L_-,L_0$ is called a \emph{Conway triple} if they are identical outside a ball and inside the ball they appear as shown in Figure \ref{skeinorient}. The crossing inside the ball can be one of two types: (1) mutual crossing between two components or (2) self crossing of a component. For type (1), we have $|L_+|=|L_0|+1$ and for type (2) we have $|L_+|=|L_0|-1$. 

	\begin{figure}[h]
	\centering
	\includegraphics[trim = 30mm 55mm 10mm 15mm, clip, width=10cm]{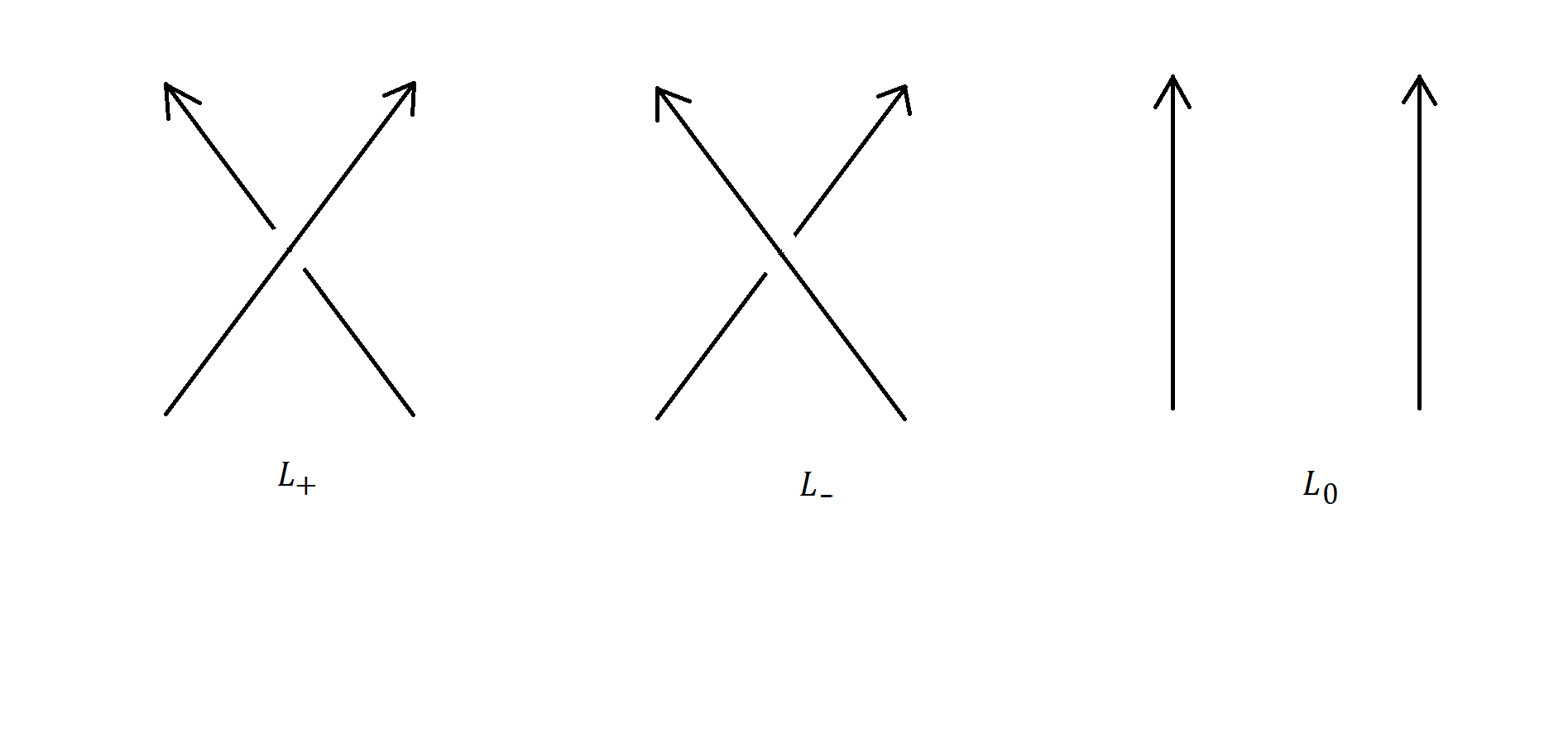}
	\caption{~}\label{skeinorient}
\end{figure}

The \emph{skein module $\ask(\Sigma)$} is a module over the polynomial ring $K[x,x^{-1},h,\hbar]$ defined as follows. Suppose $\mathscr{L}$ be the set of all isotopy classes of oriented links in $\Sigma\times I$. Then $\ask(\Sigma)$ is the quotient of the free $K[x,x^{-1},h,\hbar]$-module generated by $\mathscr{L}$ by the submodule generated by the following relations.  \\
(i)  For Conway triple with crossing type (1) we have the relation $$xL_+-x^{-1}L_--hL_0.$$ 
(ii) For Conway triple with crossing type (2) we have the relation $$xL_+-x^{-1}L_--\hbar L_0.$$


\subsection{Algebra structure on $\ask(\Sigma)$}  We fix the product orientation on $\Sigma\times I.$ We define the product $LL'$ of two links $L$ and $L'$ to be the link $L\cup L'$ obtained by stacking $L'$ above $L$. This product induces an associative algebra structure in $\ask(\Sigma)$ 
with the class of empty set being identity.

\subsection{Skein algebra {\bf A}($\Sigma$)} The skein algebra {\bf A}($\Sigma$) is defined to be the quotient of $\ask(\Sigma)$ by the ideal $(x-1)\ask(\Sigma)$. Therefore {\bf A}($\Sigma$) is an associative algebra over the polynomial ring $K[h,\hbar]$. 

\begin{theorem}\label{thm:A}
	The center of the skein algebra {\bf A}($\Sigma$)/$\hbar$ {\bf A}($\Sigma$) over $K[h]$ is generated by the empty link, the constant link and the  links which are isotopic to the boundary components or punctures of $\Sigma$.
\end{theorem}
\begin{proof}
	The proof follows from \cite[Theorem 4.2]{turaev1991skein} and Corollary \ref{cor:V}. 
\end{proof}
\begin{theorem}
	The center of the skein algebra $\ask(\Sigma)/((x-1)\ask(\Sigma)+(h-1)\ask(\Sigma)+\hbar \ask(\Sigma))$ over $K$
	is generated by the empty link, the constant link and the  links which are isotopic to the boundary components or punctures of $\Sigma$.
\end{theorem}
\begin{proof}
	The proof follows from \cite[Corollary 4.5]{turaev1991skein} and Theorem \ref{thm:univorient}.
\end{proof}
\begin{theorem}
	The center of the skein algebra $\ask(\Sigma)/\hbar \ask(\Sigma)$ over $K[x,h]/(x^2-1)$
	is generated by the empty link, the constant link and the  links which are isotopic to the boundary components or punctures of $\Sigma$.
\end{theorem}
\begin{proof}
	The proof follows from \cite[Section 4.7]{turaev1991skein} and Corollary \ref{cor:V}. 
\end{proof}
\section{Skein algebras of unoriented curves}\label{sec:unorient}
 Let $R$ be the commutative ring $K[x,x^{-1},h_{-1},h_0,h_1]/(h_0^2-h_{-1}h_1).$  Consider the three manifold $\Sigma\times I$. Let $\mathscr{L}_\Box $ be the set of all regular isotopy classes of unoriented link diagrams. Recall that the regular isotopy is the equivalence relation in the link diagrams generated by 2nd and 3rd Reidemeister moves only. We denote the number of components of a link diagram $D$ by $|D|$. 

The \emph{skein module $\kau(\Sigma)$} is a module over $R$ which is defined to be the quotient of the free $R$-module generated by $\mathscr{L}_\Box$ by the submodule generated by the following relations.  \\
i) $D_+-D_--h_{|D_+|-|D_0|}D_0+h_{|D_+|-|D_\infty |}D_{\infty},$
   where $D_+,D_-,D_0,D_\infty$ are arbitrary set of four non-empty link diagrams which are identical except in the neighbourhood of one crossing where they appear as shown in Figure \ref{skeinunorient} (replace $L$ by $D$).\\
ii) $D'-xD$, where $D'$ and $D$ are arbitrary pair of non-empty link diagrams which are identical except in the neighbourhood of one crossing where they appear as shown in Figure \ref{turaev1}.\\
iii) $h_{-1}\mathcal{O}-(x-x^{-1}+h_0)\Phi$, where $\mathcal{O}$ denotes the link diagram of the trivial knot and $\Phi$ denotes the link diagram of the empty link. 

\begin{figure}[h]
	\centering
	\includegraphics[trim = 0mm 55mm 60mm 20mm, clip, width=6cm]{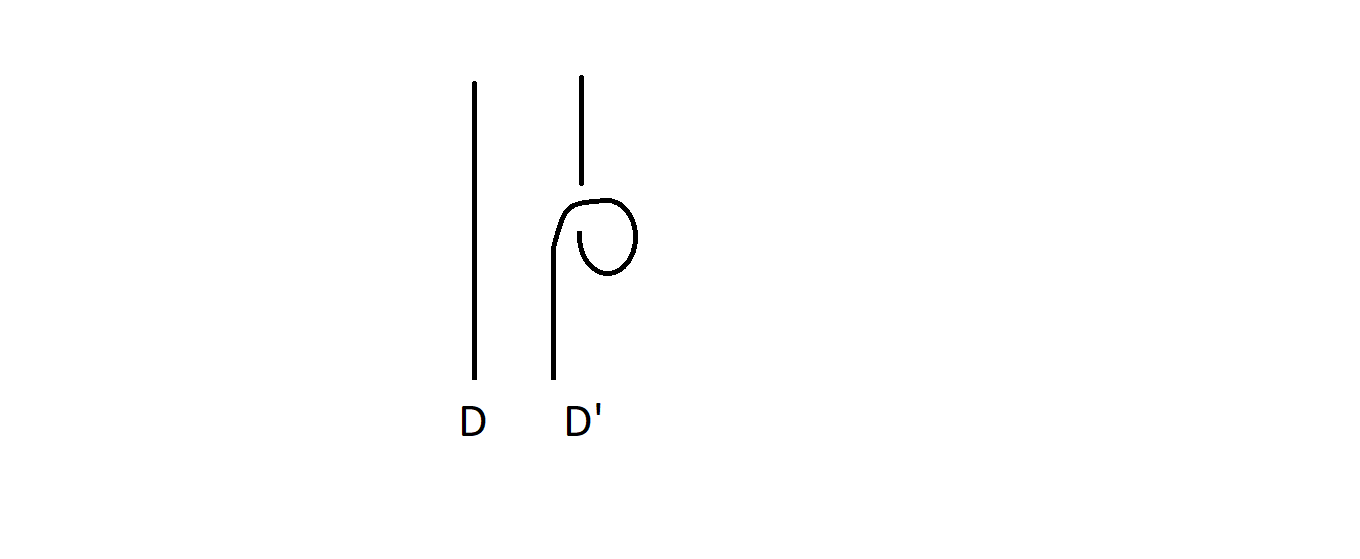}
	\caption{~}\label{turaev1}
\end{figure}

As in the oriented case, we define the product structure on $\kau(\Sigma)$ by stacking one link over another which makes $\kau(\Sigma)$ an associative algebra.

\begin{theorem}
	The center of the skein algebra $\kau(\Sigma)/((x-1)\kau(\Sigma)+h_{-1}\kau(\Sigma) )$ over $K[h_1]$ is generated by the empty link, the constant link and the  links which are isotopic to the boundary components or punctures of $\Sigma$.
\end{theorem}  
\begin{proof}
	The proof follows from \cite[Theorem 6.4]{turaev1991skein} and Theorem \ref{thm:univunorient}.
\end{proof}
\begin{theorem}
	The center of the skein algebra $\kau(\Sigma)/h_0\kau(\Sigma)+h_{-1}\kau(\Sigma) )$ over $K[x,h_1]/(x^2-1)$ is generated by the empty link, the constant link and the  links which are isotopic to the boundary components or punctures of $\Sigma$.
\end{theorem} 
\begin{proof}
	The proof follows from the first equation of \cite[Page 653]{turaev1991skein} and Corollary \ref{cor:V}.
\end{proof}
\section{Homotopy skein algebras of $\Sigma\times I$}\label{sec:homskein}
In \cite{hoste1990homotopy}, \emph{homotopy skein modules} for 3-manifolds were introduced. In this section, we recall its definition and its relation with Goldman Lie algebras $\G$ and $\GW$. For this section we assume $K$ to be $\Z[a]$, the integral polynomial ring in variable $a$.

Two links are said to be \emph{link homotopic} in $\Sigma\times I$ if one can be deformed to another by isotopy and componentwise homotopy, i.e. each component is allowed to cross itself but crossing between two distinct component is not allowed. 

 Let $\mathscr{L}_h$ (respectively $\conju{\mathscr{L}}_h$) be the set of all link homotopy classes of oriented (respectively unoriented) links in $\Sigma\times I$, including the empty link. Let $\M(\mathscr{L}_h)$ (respectively $\M\conju{\mathscr{L}}_h$) the free $K$ module generated by $\mathscr{L}_h$ (respectively $\conju{\mathscr{L}}_h$).

\begin{figure}[h]
	\centering
	\includegraphics[trim = 0mm 55mm 10mm 0mm, clip, width=12cm]{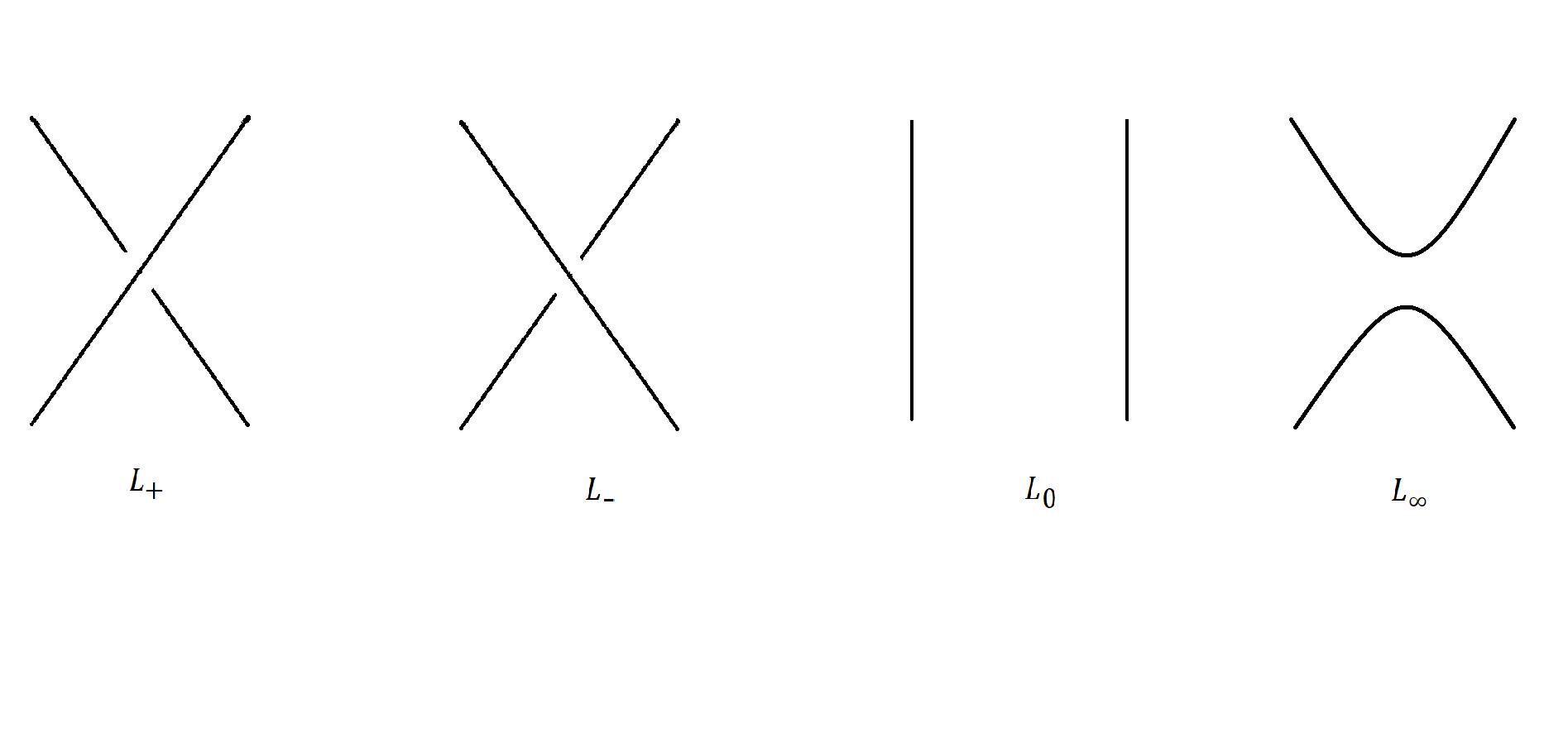}
	\caption{~}\label{skeinunorient}
\end{figure}
Suppose $L_+,L_-$ and $L_0$ are three oriented links which are identical outside a ball and inside the ball they look as in Figure \ref{skeinorient}. Also assume that the two arcs belong to two different components. Let $S(\mathscr{L}_h)$ be the submodule of $\M(\mathscr{L}_h)$ generated by all the relations of the form $L_+-L_--zL_0$. We define the \emph{homotopy skein module} $\mathscr{H}\mathscr{S}(\Sigma\times I)$ to be $\M(\mathscr{L}_h)/S(\mathscr{L}_h).$   

Similarly suppose $L_+,L_-,L_0$ and $L_\infty$ are four unoriented links which are identical outside a ball and inside the ball they look as in Figure \ref{skeinunorient}. Also assume that the two arcs belong to two different components. Let $S(\conju{\mathscr{L}}_h)$ be the submodule of $\M(\conju{\mathscr{L}}_h)$ generated by all the relations of the form $L_+-L_--z(L_0-L_\infty)$. We define the \emph{Kauffman homotopy skein module} $\mathscr{KH}\mathscr{S}(\Sigma\times I)$ to be $\M(\conju{\mathscr{L}}_h)/S(\conju{\mathscr{L}}_h).$

As before, the  $K$ modules $\mathscr{H}\mathscr{S}(\Sigma\times I)$ and $\mathscr{KH}\mathscr{S}(\Sigma\times I)$
admit a natural $K$ algebra structure from \emph{stacking product}. Let $L_1$ and $L_2$ be two oriented (respectively unoriented) links. Define the stacking product $L_1.L_2$ to be the link obtained by placing $L_1$ above $L_2$ (i.e. $L_2\subset \Sigma\times[0,\frac{1}{2}]$ and $L_1\subset \Sigma\times[\frac{1}{2},1]$). 

Consider the Lie brackets $[x,y]_a=a[x,y]$ and $[\conju{x},\conju{y}]_a=a[\conju{x},\conju{y}]$ in $K\pi$ and $K\conju{\pi}$ respectively. Let $\G_a$ (respectively $\GW_a)$ be the Lie algebra $\G$ (respectively $\GW$) with the Lie bracket  $[x,y]_a$ (respectively $[\conju{x},\conju{y}]_a$). It is clear that both $\G_a$ (respectively $\GW_a)$ and $\G$ (respectively $\GW)$ have the same center.  

\begin{theorem}\label{thm:skeincenter}
	The center of $ \mathscr{H}\mathscr{S}(\Sigma\times I)$ (respectively $\mathscr{KH}\mathscr{S}(\Sigma\times I)$) is generated by the empty link, the constant link and the oriented (respectively unoriented) links which are link homotopic to the boundary components or punctures of $\Sigma$.
\end{theorem}
\begin{proof} The proof follows directly from  \cite[Theorem 3.3 and Page 484]{hoste1990homotopy}, Theorem \ref{thm:univorient} and Theorem \ref{thm:univunorient}.
\end{proof}	

\section{Observations}\label{sec:spec} By \cite[Theorem 5.4 and Theorem 5.13]{goldman_invariant_1986}, we can treat the elements of $\G$ and $\GW$ as classical observables on the moduli space of representations. By the results of \cite{turaev1991skein} and \cite{hoste1990homotopy} we can treat the elements on various skein algebras discussed earlier as quantum observables. The degeneration from quantum objects to classical object is simply the natural projection of the links in $\Sigma\times I$ to curves in $\Sigma$. 

For this discussion, we fix a metric in $\Sigma$ and consider geodesics with respect to this metric without specifying it explicitly. Given an isotopy class of a link $L$, we denote its representative in $\ask$ by $[L]$. Given any link $L$ and a crossing $p_i$ of crossing type $(1)$, let $[L^{p_i}_{\e_i}]$ be the element in $\ask$ which is the term corresponding to the symbols $+,-$ or $\infty$ in Figure \ref{skeinorient} depending on whether $\e_i$ is $+,-$ or $\infty$. Given $k$ crossings $p_1,p_2,\ldots ,p_k$ of $L$ of crossing type $(1)$ and $k$ symbols $\e_1,\e_2,\ldots ,\e_k\in\{+,-,\infty \}$  define $[L^{p_1,p_2,\ldots ,p_k}_{\e_1,\e_2,\ldots \e_k}]$ inductively. We call the crossings of Figure \ref{skeinorient} associated with the symbol $+$ and $-$ as the crossings of type $+$ and $-$ respectively.  
   
Let $\pa:\Sigma\times I\rightarrow \Sigma$ be the projection map. Given two elements $[L_1], [L_2]\in \ask$, choose representatives $L_1$ and $L_2$ such that $\pa(L_1)$ and $\pa(L_2)$ are geodesics in $\Sigma$. Let $p_1,p_2,\ldots ,p_r$ be the mutual crossings between $L_1$ and $L_2$ of type $+$ and $q_1,q_2, \ldots ,q_s$ be the mutual crossings between  $L_1$ and $L_2$ of type $-$. Therefore we use the relation associated to Conway Triples with crossing type (1) at these crossings to get the following. Let $L_1L_2=L$. 
\begin{align*}
    x[L_1L_2] & =x^{-1}[L^{p_1}_-]+h[L^{p_1}_0]\\
    	       & =x^{-2}x [L^{p_1}_-]+ h[L^{p_1}_0]\\
    	       & =x^{-2}(x^{-1}[L^{p_1,p_2}_{-,-}]+h[L^{p_2}_{-0}])+ h[L^{p_1}_0]\\ 
    	       & =x^{-3}[L^{p_1,p_2}_{-,-}]+ h(x^{-2}[L^{p_2}_{-0}]+[L^{p_1}_0])\\
    	       & \vdots\\
    	       & =x^{-2r+1}[L^{p_1,p_2,\ldots ,p_r}_{-,-,\ldots ,-}]+h([L^{p_1}_0]+x^{-2}[L^{p_2}_{-0}]+\ldots +x^{-2r}[L^{p_r}_0])   
\end{align*}
Let $M=L^{p_1,p_2,\ldots ,p_r}_{-,-,\ldots ,-}$ and $T= h([L^{p_1}_0]+x^{-2}[L^{p_2}_{-0}]+\ldots +x^{-2r}[L^{p_r}_0])$. Then
\begin{align*}
		x[L_1L_2] & = [M]+T\\
		       & = x^{-2r+2}. x^{-1}[M]+T\\
		       & = x^{-2r+2}([M^{q_1}_{+}]-h[L^{q_1}_0])+T\\
		       & = x^{-2r+3}[M^{q_1}_{+}]-x^{-2r+2}h[L^{q_1}_0]+T\\
		       & \vdots \\
		       & =x^{-2r+2s+1} [M^{q_1,q_2,\ldots q_s}_{+,+,\ldots ,+}]-h(x^{-2r+2}[L^{q_1}_0]+ x^{-2r+4}[L^{q_2}_{0}]+\ldots +x^{-2r+2s}[L^{q_s}_0])+T\\ 
	       \end{align*}
 But $[M^{q_1,q_2,\ldots q_s}_{+,+,\ldots ,+}]=[L^{p_1,\ldots ,p_r,q_1,\ldots q_s}_{-,\ldots ,-,+,\ldots ,+}]=[L_2L_1].$ Putting the values of $M$ and $T$ in the equation we get 
 \begin{align*}
 	x[L_1L_2]-x^{-2r+2s+1}[L_2L_1]&=h([L^{p_1}_0]+x^{-2}[L^{p_2}_{0}]+\ldots +x^{-2r}[L^{p_r}_0]\\
 & ~~~~~~~~~~ -x^{-2r+2}[L^{q_1}_0]- x^{-2r+4}[L^{q_2}_{0}]-\ldots -x^{-2r+2s}[L^{q_s}_0])     
 \end{align*} 
By \cite[Lemma 4.3]{turaev1991skein}, $\pa([L^{p_i}_0])=\pa(L_1)*_{p_i}\pa(L_2)$ and  $\pa([L^{q_i}_0])=\pa(L_1)*_{q_i}\pa(L_2)$. 

If we choose $L_2$ such that $\la\pa(L_2)\ra=\la x^n\ra$ for some simple closed curve $x$ then by the proof of Theorem \ref{thm:univorient}, for sufficiently large $n$, the links  $L^{p_1}_0,\ldots  ,L^{p_1}_0,L^{q_1}_0, \ldots L^{q_1}_0 $ are pairwise non isotopic as the curves $\pa(L^{p_1}_0),\ldots  ,\pa(L^{p_r}_0),\pa(L^{q_1}_0), \ldots ,\pa(L^{q_s}_0) $ are pairwise not freely homotopic to each other. Therefore a necessary condition for the center of $\ask$ or any of its quotient to have elements in the Poisson center different from the classical one is that the elements corresponding to $L^{p_1}_0,\ldots  ,L^{p_r}_0,L^{q_1}_0, \ldots L^{q_s}_0$ must be related to each other via the skein relations. In other words if the skeins appearing on the right hand side of the above equation are linearly independent in  $\ask$ or any of its quotients then their Poisson centers are the algebra generated by  links isotopic to boundary or punctures. The process also works in case of skein algebras of unoriented loops. 
       
	\bibliographystyle{amsplain}
	
	\bibliography{gla.bib}
\end{document}